\documentclass[12pt,a4paper]{article}
\usepackage{amsmath}
\usepackage{amssymb}
\usepackage{accents}
\usepackage{hyperref}
\usepackage{mathtools}
\usepackage{amsthm}
\usepackage{xcolor}
\usepackage[a4paper,portrait, margin=15mm]{geometry}
\usepackage{graphicx}
\usepackage{booktabs}
\usepackage{comment}

\newcommand{\eps}{\varepsilon}
\newcommand{\Z}{{\mathbb Z}}

\newcommand{\R}{{\mathbb R}}

\newcommand{\IP}{{\mathbb P}}
\newcommand{\IE}{{\mathbb E}}
\newcommand{\cR}{{\mathcal R}}

\allowdisplaybreaks

\newtheorem{theo}{Theorem}[section]
\newtheorem{lem}[theo]{Lemma}
\newtheorem{df}[theo]{Definition}

\newtheorem{remark}[theo]{Remark}

\title{On extinction and survival of the discrete Bak--Sneppen model on arbitrary graphs}
\author{Serguei Popov$^{1}$ \and Stanislav Volkov$^{2}$}
\begin{document}
\maketitle

{\footnotesize
\noindent $^{~1}$Centro de Matem\'atica, University of Porto, Porto, Portugal\\
\noindent e-mail: \texttt{serguei.popov@fc.up.pt}

\noindent $^{~2}$ Lund University, Centre for Mathematical Sciences,
Lund, Sweden\\ \noindent e-mail: \texttt{stanislav.volkov@matstat.lu.se}}

\begin{abstract}
We study the stationary distribution of the discrete Bak--Sneppen model on finite connected graphs of bounded degree, generalizing the model introduced in~\cite{BK}. In the small-$p$ regime, a graphical construction and a comparison with high-density one-dependent oriented site percolation yield a bound of order $p\ln(1/p)$ on the stationary probability that a prescribed vertex has fitness one, uniformly over all such graphs. If the graph can be covered by long chains with bounded overlap, we also obtain an exponential upper tail bound for the stationary density of ones. In the complementary regime, a Lyapunov function argument yields an exponential stationary tail for the number of zeros, uniformly over graphs of a given maximal degree, with explicit parameter bounds. 
\\[.3cm]\textbf{Keywords:} Bak--Sneppen model, self-organized criticality, oriented percolation.
\\[.3cm]\textbf{MSC2020 subject classifications: 60J20, 60K35, 82C22}
\end{abstract}

\section{Introduction and main results}
\label{s_intro}
The classical Bak--Sneppen model, introduced in~\cite{BAK,BAKS}, is defined as follows. There are $N\ge 3$ points uniformly spread over a circumference. Each point at discrete time $t=0,1,2,\dots$ possesses a value $\eta_i(t)\in [0,1]$, $i\in [N]=\{1,2,\dots,N\}$, called {\em fitness}; we assume that at time zero all fitnesses are distinct. At each time~$t$ we choose the index $j=j(t)$ such that
\[
\eta_j(t)=\min_{i\in[N]} \eta_i(t)
\]
and replace the fitness of site~$j$ as well as that of both its nearest neighbours on the circumference by three uniform $[0,1]$ values, drawn independently. Due to continuity of these random variables, we will a.s.\ never have ties and~$j$ is thus uniquely defined. Formally,
\[
\eta_i(t+1)=\begin{cases}
 \zeta_{i-j,t}, &\text{if } i\in \{j-1,j,j+1\};\\
 \eta_i(t), &\text{otherwise}
\end{cases}
\]
with the convention $N+1=1$, $1-1=N$,  where $\zeta_{-1,t},\zeta_{0,t},\zeta_{+1,t}$ are i.i.d.\ uniform $[0,1]$.
A long-standing conjecture says that as $N\to\infty$, the asymptotic distribution of $\eta_i(t)$ converges to the uniform distribution over the interval $[\eta_c,1]$, where $\eta_c\approx 0.66$. This has recently been proved in~\cite{Henard}, building on~\cite{MZ04}, in the stronger form that the fitnesses at any fixed finite set of sites are asymptotically independent and uniform on $[\eta_c,1]$; the exact value of~$\eta_c$ remains unknown.

A {\em discrete} version of the Bak--Sneppen model introduced in~\cite{BK} and later studied by~\cite{MS,MZ,VOLK} allows just two values of fitnesses, namely $\eta_i(t)\in\{0,1\}$, and the fitnesses are drawn from the $\mathrm{Bernoulli}(p)$ distribution. In this case the minimum value is typically not unique, and one draws $j(t)$ with equal probability among all those $j$ such that $\eta_j(t)=\min_i \eta_i(t)$ (this minimum is $0$, except when all $\eta_i$ are ones). When $0<p<1$, we obtain an ergodic Markov chain with $2^N$ states. Let
\[
\pi_{N,p}(\eta: \eta_i=1)=\lim_{t\to\infty}\IP[\eta_i(t)=1]
\]
be the probability\footnote{On the circle (and, more generally, on vertex-transitive graphs) this does not depend on~$i$ by symmetry.} that site $i$ is in state $1$ under the stationary measure. We are interested in describing the above limit for very large $N$.

On the circle, by comparison with a branching process, it is easy to see that if $p> 2/3$ then $\pi_{N,p}(\eta: \eta_i=1)\to 1$ as $N\to\infty$.
It is much harder to show that the same does not happen for {\em all} $p>0$. This was eventually shown in~\cite{MZ} for small enough $p$; the proof in~\cite{MZ} is quite complicated. The first purpose of the current paper is to provide an alternative, and possibly simpler, proof, based on a comparison with oriented site percolation, which in addition gives the quantitative bound $\pi_{N,p}(\eta: \eta_i=1)=O(p\ln p^{-1})$ and extends to arbitrary finite connected graphs of bounded degree $G$, defined as follows: one picks uniformly one of the vertices with minimal fitness and replaces its fitness, together with the fitnesses of all its neighbours in~$G$, by independent Bernoulli($p$) variables.
\vspace{3mm}

Throughout the paper we will assume that~$G$ is a finite connected graph with at least two vertices, without loops or multiple edges, and $d\ge 2$ is an upper bound for its maximal degree. For simplicity, we do not distinguish between~$G$ and its set of vertices; we write $x\sim y$ if~$x$ and~$y$ are neighbours in~$G$; for $x\in G$, we write $N(x)=\{y:y\sim x\}$ and call $N[x]=\{x\}\cup N(x)$ the \emph{neighbourhood} of~$x$. We also denote $q:=1-p$.

It is convenient to also consider the continuous-time version of the model: each vertex $x$ has its own independent Poisson clock of rate~$1$. When the clock of~$x$ rings, if the fitness of $x$ is $0$, then the fitness of $x$ and of each of its neighbours $y\sim x$ is replaced by an independent Bernoulli($p$) variable. 
If the fitness of~$x$ is~$1$, nothing happens, unless the configuration contains no zeros; in that case, the same update is performed.  Call a ring of the clock of~$x$ \emph{active} if~$x$ has fitness~$0$, or if the configuration is all ones. The chain observed at the successive active rings (including active rings whose resampling happens to leave the configuration unchanged) is precisely the discrete-time Bak--Sneppen chain described above. The graphical construction of Section~\ref{s_blocks} would also allow us to define the continuous-time dynamics on infinite graphs of bounded degree, but in this paper we only deal with finite graphs.

We denote by $\pi^c$ the stationary distribution of the continuous-time process and by $\pi^d$ the stationary distribution of the discrete-time chain. Even though they are not the same, they are related by a simple state-dependent time change, as shown in Appendix~\ref{s_app_measures}. Theorems~\ref{t_small_p_1_in_x} and~\ref{t_small_p_proportion} are proved for $\pi^c$ and transferred to $\pi^d$ at the cost of a factor of at most~$2$, while Theorem~\ref{t_extinct_zeros} is proved directly for~$\pi^d$ by a Lyapunov function argument, and transferred to~$\pi^c$ at the cost of a multiplicative constant, see Remark~\ref{r_pic_zeros}.

There are also numerous results for related models. In~\cite{Guiol} the authors studied a model of evolution without spatial structure. The Bak--Sneppen evolution model on a regular hypercubic lattice in high dimensions was studied in~\cite{Chhimpa}. A local barycentric version of the process was studied in~\cite{VOLK2}. The hydrodynamic limit of the Bak--Sneppen branching diffusion was obtained in~\cite{Ilie1,Ilie2}. For other recent and related results on the Bak--Sneppen model and its simulations please see~\cite{BB,BR,GKW,GRAS,FRA,MS} and references therein.

A particularly relevant result is that of Bandt~\cite{BANDT}, who showed that the discrete Bak--Sneppen model is conjugate to a version of the contact process; this holds on every finite graph and for several versions of the two dynamics (parallel or sequential updates, all neighbours or one random neighbour, discrete or continuous time). Let us state the version relevant here in our notation (in~\cite{BANDT} the roles of the states $0$ and~$1$ are interchanged, and Bandt's parameter~$p$ is our~$q$). Let $G$ be finite, and for $v\in G$ let $D_v$ be the Bak--Sneppen update at~$v$ \emph{without regeneration}: if $\eta_v=0$, the fitnesses on $N[v]$ are resampled, and otherwise nothing happens, so that the all-ones configuration is absorbing. Let $C_v$ be the following contact process update at~$v$, where zeros play the role of infected vertices: if $\eta_v=0$, then with probability~$q$ all vertices of $N[v]$  get fitness~$0$, and with probability~$p$ the vertex $v$ gets fitness~$1$; otherwise nothing happens. Finally, let~$T$ be the \emph{thinning} operator, which independently keeps every zero with probability~$q$ and turns it into a one with probability~$p$, and leaves ones unchanged. For this version, \cite{BANDT} proves the relation $D_vT=TC_v$ for every~$v$ (as Markov kernels, $T$ being applied first on the left-hand side), and hence $DT=TC$ for $D=|G|^{-1}\sum_vD_v$ and $C=|G|^{-1}\sum_vC_v$: running the contact process from~$\eta$ and then thinning gives the same law as thinning~$\eta$ and then running the (absorbing) discrete Bak--Sneppen dynamics. Summing over~$v$, we also obtain $\mathcal L_DT=T\mathcal L_C$ for the generators $\mathcal L_D=\sum_v(D_v-I)$ and $\mathcal L_C=\sum_v(C_v-I)$, which is the corresponding relation in continuous time (in~\cite{BANDT}, the continuous-time case is treated for the one-random-neighbour versions); 
note that $\mathcal L_D$ is the generator of our continuous-time process with the regeneration switched off. As a corollary of this relation,  starting from a single zero, the probability that the zeros survive up to a given time is, for the Bak--Sneppen dynamics, at least the corresponding probability for the contact process and at most $q^{-1}$ times it.\footnote{In~\cite{BANDT} this corollary is formulated for the probabilities of having died out; however, the proof given there concerns the survival probabilities, and the ratio bound does not hold for the extinction probabilities.} 
By a standard approximation by finite subgraphs, this comparison extends to infinite graphs of bounded degree, where it implies that the two absorbing processes have the same critical parameter. The contact process~$C$ appearing here is the ``all-neighbour'' version, in which an infected vertex infects its whole neighbourhood at once; it is not the classical nearest-neighbour contact process, so quantitative results for the latter (such as rigorous bounds on its critical value) do not transfer directly.

Our focus in the current paper is, however, different. On a finite graph, the Bak--Sneppen chain continues to evolve after reaching the all-ones configuration (all vertices are then minimizers, one of them is chosen uniformly, and its neighbourhood is resampled), so that the resulting finite chain is {\em ergodic}, hence not absorbing. We study the invariant distribution of this finite chain; our estimates (a one-site bound of order $p\ln p^{-1}$ uniformly over all graphs of bounded degree, density bounds and exponential tails) are formulated for this invariant law.  In particular, we are not aware of results on the all-neighbour contact process from which Theorems~\ref{t_small_p_1_in_x}--\ref{t_extinct_zeros} below could be derived via~\cite{BANDT}.

\begin{df}
\label{df_chain}
We say that $(x_1,\ldots,x_n)$ is a \emph{self-avoiding path}, if $x_j\sim x_{j+1}$ for all $j=1,\ldots,n-1$, and all sites in the path are distinct. Further, we say that a self-avoiding path $(x_1,\ldots,x_n)$ is a \emph{chain}, if the neighbourhoods $N[x_i]$ and $N[x_j]$ do not intersect whenever $j-i\geq 3$. The \emph{length} of a chain is the number~$n$ of its vertices.
\end{df}

Any sub-path $(x_i,\dots,x_j)$ of a chain is again a chain. Let $\ell(G)$ be the maximal length of a chain in~$G$. A shortest path $(x_1,\dots,x_n)$ between two sites is a chain: if $j-i\ge3$ and $z\in N[x_i]\cap N[x_j]$, then the graph distance between~$x_i$ and~$x_j$ would be at most $2<j-i$. Since a shortest path between two sites at distance $\mathrm{diam}(G)$ has $\mathrm{diam}(G)+1$ vertices, we obtain
\begin{equation}
\label{diam_ell}
\mathrm{diam}(G)+1\le \ell(G).
\end{equation}
Our first main result says that ones become ``locally improbable'' if the parameter~$p$ is small, uniformly over all graphs of bounded degree.
\begin{theo}
\label{t_small_p_1_in_x}
For every $d\ge2$ there exist $p_0=p_0(d)>0$ and $C_d<\infty$ such that for all $p\in(0,p_0]$, all finite connected graphs~$G$ of maximal degree at most~$d$ and all $x\in G$ we have
\[
\pi^c(\eta: \eta_x=1) \le C_d\, p\ln p^{-1}
\qquad\text{and}\qquad
\pi^d(\eta: \eta_x=1) \le 2C_d\, p\ln p^{-1}.
\]
In particular, $g_d(p):=\sup_{G,x}\pi^d(\eta: \eta_x=1)\to 0$ as $p\downarrow 0$, where the supremum is over all graphs as above.
\end{theo}

\begin{remark}
\label{r_lower_bound}
The order of magnitude in Theorem~\ref{t_small_p_1_in_x} is correct up to the logarithmic factor: for every finite connected~$G$ of maximal degree at most~$d$ and every~$x$,
\begin{equation}
\label{eq_lower_bound}
\pi^c(\eta:\eta_x=1)\ge \frac{p}{d+2}.
\end{equation}
Indeed, under $\pi^c$ the rate at which $\eta_x$ jumps from $0$ to $1$ equals the rate at which it jumps from $1$ to $0$. A jump $1\to0$ of~$\eta_x$ requires a ring at some $u\in N[x]$ which is active (i.e., $\eta_u=0$, or the configuration is all ones), so its rate is at most $(d+1)\pi^c(\eta_x=1)$. On the other hand, when $\eta_x=0$ the clock of~$x$ itself is active and turns~$\eta_x$ into~$1$ with probability~$p$, so the rate of $0\to1$ jumps is at least $p\,\pi^c(\eta_x=0)$. Hence $(d+1)\pi^c(\eta_x=1)\ge p(1-\pi^c(\eta_x=1))$, which gives~\eqref{eq_lower_bound}. We do not know whether the factor $\ln p^{-1}$ in Theorem~\ref{t_small_p_1_in_x} can be removed.
\end{remark}

Next, we state a ``concentration of measure'' result: if $G$ can be covered by long chains without too much overlap, then in the small-$p$ regime the stationary probability of configurations with ``too many'' ones is exponentially small in the length of the chains. For a configuration~$\eta$, let $\bar\eta=|G|^{-1}\sum_{x\in G}\eta_x$ be the proportion of ones in~$\eta$. Exponential density estimates for supercritical contact processes and oriented percolation on one-dimensional regions are classical, see e.g.~\cite{DS1,DS2}; here we use such estimates for oriented site percolation on a strip.

\begin{theo}
\label{t_small_p_proportion}
For every $d\ge 2$ there exist $p_1=p_1(d)>0$ and constants $C,C'>0$ depending only on~$d$ such that the following holds for all $p\in(0,p_1]$. Let~$G$ be a finite connected graph of maximal degree at most~$d$, let $m\ge 1$, and let $\ell\ge 1$ be an integer. Assume that there is a finite family of chains in~$G$, each of length at least~$\ell$, such that every vertex of~$G$ belongs to at least one of them, and the sum of their lengths is at most~$m|G|$. Then
\[
\pi^c\Big(\eta: \bar\eta > \frac{mC'}{\ln p^{-1}}\Big)
 \leq Cm^2|G|^3 e^{-\ell},
\qquad
\pi^d\Big(\eta: \bar\eta > \frac{mC'}{\ln p^{-1}}\Big)
 \leq 2Cm^2|G|^3 e^{-\ell}.
\]
\end{theo}

The hypothesis holds, in particular, if every vertex of~$G$ belongs to at least one and at most~$m$ of the chains; for $m=1$ this means that the chains form a partition of the vertex set of~$G$. The bound is meaningful when $\ell\gg \ln|G|$. For instance, the box $\{1,\dots,n\}^k$, $k\ge1$, is partitioned into $n^{k-1}$ straight lines of length~$n$ parallel to the first coordinate axis, which are shortest paths and hence chains, so that the right-hand side is at most $Cn^{3k}e^{-n}$. For the torus $(\Z/n\Z)^k$ (in particular, for the cycle), such a straight line is \emph{not} a chain, since its endpoints are neighbours; however, it splits into two arcs with $\lfloor n/2\rfloor$ and $\lceil n/2\rceil$ vertices, which are shortest paths in the torus and hence chains, so that the torus is partitioned into $2n^{k-1}$ chains of length at least $\lfloor n/2\rfloor$, and the right-hand side is at most $Cn^{3k}e^{-n/2}$ (with a different~$C$). On the other hand, the hypothesis is restrictive for graphs with many vertices of degree one (such as trees): a chain contains at most two such vertices, so that $\ell\le 2m|G|/J$ if~$G$ has~$J$ of them. It would be interesting to determine what can be said in the case of more general graphs; we leave this as an open problem.

In the large-$p$ regime, equivalently for sufficiently small $q=1-p$,
we obtain a {\em quantitative} statement for the stationary law of the regenerated Bak--Sneppen chain: uniformly over finite graphs of maximal degree~$d$, the stationary number of zero-fitness vertices has an exponentially decaying tail throughout an explicit interval extending beyond the elementary one-step branching bound. 

\begin{theo}\label{t_extinct_zeros}
Recall that $q=1-p$. There is $q_0=q_0(d)>\frac{1}{d+1}$, defined in~\eqref{def_q0}, such that if $q< q_0$, then the zeros ``become extinct'', in the sense that
\begin{equation}
\label{eq_extinct}
 \pi^d(\eta: \text{there are more than $k$ zeros})\leq c_1e^{-c_2k}\quad\text{for all }k>0,
\end{equation}
where the positive constants $c_1,c_2$ depend only on~$d$ and~$q$ (and not on~$G$ directly). Moreover,~\eqref{eq_extinct} also holds with $\pi^d$ replaced by~$\pi^c$, with a different constant~$c_1$ (depending only on~$d$ and~$q$), see Remark~\ref{r_pic_zeros}.
\end{theo}

\begin{remark}
\label{r_compare_VOLK}
For the cycle ($d=2$), Theorem~\ref{t_extinct_zeros} is \emph{not} the best known result. It was shown in~\cite{VOLK} that the stationary density of ones tends to~$1$ as $N\to\infty$ for $p>0.455$, and, with a finer tuning of the method, for all $p>0.419533$, i.e.\ for $q<0.545$ and $q<0.580467$ respectively; this is considerably better than $q_0(2)\approx0.4116$. (The statements are not identical: \eqref{eq_extinct} is an exponential tail bound for the number of zeros, which is stronger than convergence of the density.) The point of Theorem~\ref{t_extinct_zeros} is that it holds uniformly over all graphs of maximal degree~$d$.
\end{remark}

Throughout the paper, $c,C,C',c_1,C_1,\dots$ denote positive constants. Their dependence on the parameters is specified when they are introduced (a constant with no specified dependence is absolute); in particular, they do not depend on the graph~$G$ unless this is explicitly indicated, as in~$C_G$.
\vspace{3mm}

The rest of the paper is organized as follows. In Section~\ref{s_proofs} we recall a flow lemma for stationary measures, describe the graphical construction, and introduce the blocks used in the comparison with oriented site percolation. The percolation estimates are proved in Section~\ref{s_OP_strip}. We prove Theorems~\ref{t_small_p_1_in_x} and~\ref{t_small_p_proportion} in Section~\ref{s_proofs_main}, and Theorem~\ref{t_extinct_zeros} in Section~\ref{s_proof_t14}. The appendix relates the measures~$\pi^c$ and~$\pi^d$.

\section{Preliminaries: a flow lemma, the graphical construction and blocks}\label{s_proofs}
We start by recalling a general (but elementary) fact about Markov chains on a finite state space, in discrete or continuous time. We denote by $P^{(t)}_{uv}$ the transition probability in time~$t$ from~$u$ to~$v$, and  let~$\pi$ denote a stationary measure.  The following fact is useful especially in the nonreversible case and formalises the intuitively clear idea that if it is easy to leave the set~$A$, but difficult to reach it when starting from outside, then the stationary probability of being in~$A$ should be small. This approach was employed, for example, in~\cite{P05}.

\begin{lem}
\label{l_est_pi}
Let $A$ be some subset of the states of the Markov chain and $A^\complement$ denote all the remaining states.
 Assume that for some $t>0$
\[
 \sum_{v\in A^\complement} P^{(t)}_{uv}\geq c>0   \qquad \text{for all $u\in A$,}
\]
 and
\[ \sum_{u\in A} P^{(t)}_{vu}\le \eps   \qquad \text{for all $v\in A^\complement$.}\]
Then $\pi(A) \le \eps/c$.
\end{lem}
\begin{proof}
By stationarity, the flow out of~$A$ during time~$t$ equals the flow into~$A$; the former is at least $c\,\pi(A)$, and the latter is at most $\eps\,\pi(A^\complement)\le\eps$.
\end{proof}

\subsection{Graphical representation}
\label{s_blocks}
We use the following graphical construction of the continuous-time process defined in Section~\ref{s_intro} (see Figure~\ref{f_graphical_constr}, which illustrates it in the case of $\mathbb{Z}$). Each ring of the clock of~$x$ carries a \emph{mark} $\omega=(\omega_y)_{y\in N[x]}\in\{0,1\}^{N[x]}$, whose coordinates are i.i.d.\ Bernoulli($p$); all clocks and marks are independent. If a ring of~$x$ at time~$\tau$ is active, we set $\eta_y(\tau)=\omega_y$ for all $y\in N[x]$, and otherwise nothing happens. We will constantly use the following obvious observation:
\begin{itemize}
\item[(O)] the fitness $\eta_v$ can change only at active rings of vertices $u\in N[v]$, and at such a ring it becomes equal to the $v$-coordinate of the mark.
\end{itemize}
We say that a ring \emph{carries a~$1$ at~$y$} if its mark has $\omega_y=1$ (so it is a ring of a vertex of~$N[y]$). For $y\in G$ and a time interval~$I$, the number of rings in~$I$ carrying a~$1$ at~$y$ is a Poisson random variable with mean $|N[y]|\,p\,|I|\le(d+1)p|I|$. In particular, by~(O), if~$y$ has fitness~$0$ at the beginning of~$I$, then it keeps fitness~$0$ during~$I$ with probability at least $1-(d+1)p|I|$.

\begin{figure}
\begin{center}
\includegraphics[scale=0.9]{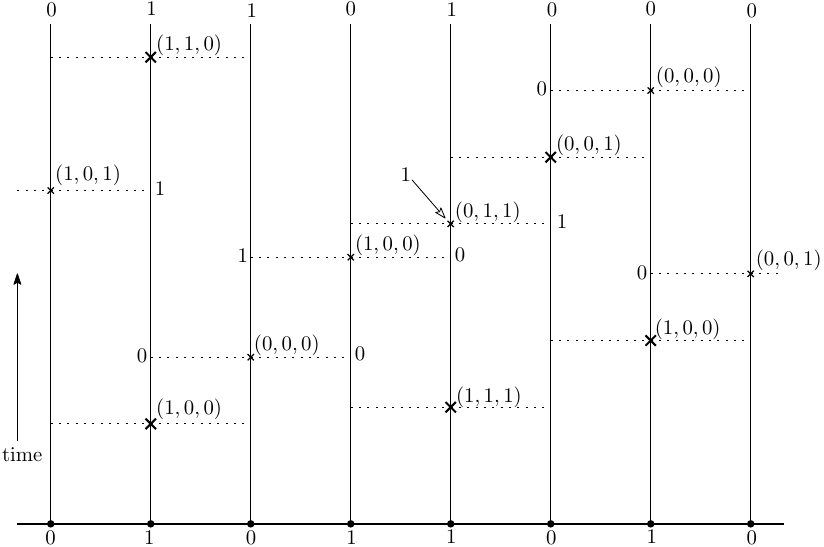}
\caption{The graphical construction: each site has a marked Poisson clock of rate one, with the marks being the proposed neighbourhood changes. When the clock rings and the current state is~$1$, nothing happens (unless the configuration is all ones); if the current state is~$0$, the changes take place. The rings that are ``muted'' (because the current state is~$1$) are indicated by larger crosses on the picture.}
\label{f_graphical_constr}
\end{center}
\end{figure}

\subsection{Blocks}
\label{s_4blocks}
Our method of analyzing the Bak--Sneppen model in the small-$p$ regime is based on a comparison with oriented site percolation. The basic objects are blocks built on top of pairs of consecutive vertices of a chain.

Fix a chain $(z_0,\dots,z_r)$, $r\ge1$, and write $S_i:=\{z_i,z_{i+1}\}$ for $i=0,\dots,r-1$; note that~$S_i$ and~$S_{i+1}$ have exactly one common vertex. Fix also a starting time~$s$ and a level length $L>0$. For $i\in\{0,\dots,r-1\}$ and $j\in\Z_+$ with $i+j$ even, the \emph{block} $B(i,j)$ is $S_i\times I$, where $I=(s+jL,s+(j+1)L]$; we say that it is at \emph{level}~$j$ (see Figure~\ref{f_blocks}). We declare the block $B(i,j)$ \emph{nice}, if  
\begin{itemize}
\item[(N1)] no ring in~$I$ carries a~$1$ at a vertex of~$S_i$;
\item[(N2)] the clocks of both vertices of~$S_i$ ring in~$I$.
\end{itemize}

\begin{figure}
\begin{center}
\includegraphics{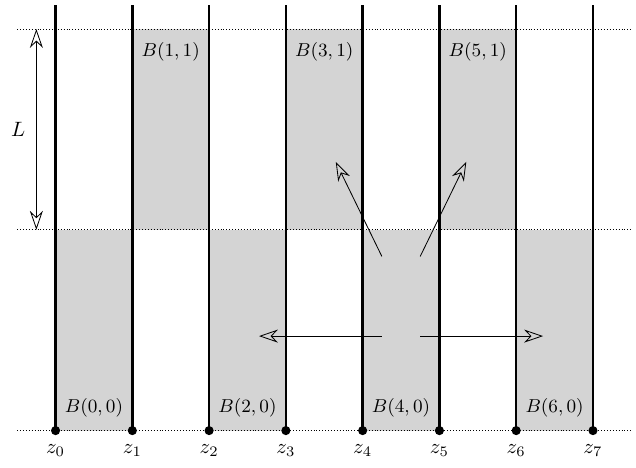}
\caption{The blocks at levels~$0$ and~$1$ for a chain $(z_0,\dots,z_7)$. For the block $B(4,0)$, the horizontal arrows point to the two blocks of the same level which are not independent of it, and the diagonal arrows to the blocks of its children.}
\label{f_blocks}
\end{center}
\end{figure}

\begin{lem}
\label{l_prop_block}
If the block $B(i,j)$ is nice and at time $s+jL$ at least one of the two vertices of~$S_i$ has fitness~$0$, then both of them have fitness~$0$ at time $s+(j+1)L$.
\end{lem}
\begin{proof}
By~(O) and~(N1), during~$I$ no vertex of~$S_i$ can change its fitness from $0$ to $1$. Now let $S_i=\{x,y\}$ and assume that~$x$ has fitness~$0$ at time $s+jL$. Then~$x$ has fitness~$0$ during all of~$I$, so its first ring in~$I$ (which exists by~(N2)) is active and, by~(N1), sets $\eta_y=0$; by the above, $y$ then keeps fitness~$0$ until the end of~$I$.
\end{proof}

Blocks of the same level are not independent; however, we have the following.
\begin{lem}
\label{l_blocks_indep}
Let~$\mathcal B$ be a finite set of blocks.
\begin{itemize}
\item[(i)] If $|i-i'|\ge4$ for any two distinct blocks $B(i,j)$, $B(i',j)$ of~$\mathcal B$ which are at the same level, then the events $\{B\text{ is nice}\}$, $B\in\mathcal B$, are independent.
\item[(ii)] In any case, $\mathcal B$ contains at least $\lceil|\mathcal B|/2\rceil$ blocks which satisfy the assumption of~(i).
\end{itemize}
\end{lem}
\begin{proof}
The event $\{B(i,j)$ is nice$\}$ depends only on the rings in $(s+jL,s+(j+1)L]$ of the clocks of the vertices of $N[z_i]\cup N[z_{i+1}]$ and on their marks. Blocks at different levels use disjoint time intervals. If $|i-i'|\ge4$, then any vertex of~$S_i$ and any vertex of~$S_{i'}$ have indices differing by at least~$3$ along the chain, so that their neighbourhoods are disjoint by the definition of a chain; thus $B(i,j)$ and $B(i',j)$ depend on disjoint sets of clocks, which proves~(i). For~(ii), order the blocks of~$\mathcal B$ at each level according to~$i$ and keep every second one, starting with the first; since the values of~$i$ at a given level have the same parity, they differ by at least~$4$ for any two of the blocks we keep.
\end{proof}

Let us estimate the probability that a block is nice. From now on we fix
\begin{equation}
\label{def_L_eps}
L=L_p:=\ln p^{-1},\qquad
\eps_d(p):= 2(d+1)p\ln p^{-1}+2p .
\end{equation}
\begin{lem}
\label{l_block_nice_prob}
For every block, $\IP[\text{the block is nice}]\ge 1-2(d+1)pL-2e^{-L}$. In particular, with $L=L_p$ we have $\IP[\text{the block is nice}]\ge 1-\eps_d(p)$.
\end{lem}
\begin{proof}
By the observation after~(O), (N1) fails with probability at most $2(d+1)pL$, and each of the two clocks has no ring in~$I$ with probability~$e^{-L}$; with $L=L_p$ we have $e^{-L}=p$.
\end{proof}

\paragraph{The comparison with oriented site percolation.} Consider the oriented graph with vertex set 
\begin{equation}\label{Ldefined}
\mathbb L_{r-1}=\{(i,j): 0\le i\le r-1,\ j\in\Z_+,\ i+j\text{ even}\}    
\end{equation}
and oriented edges from $(i,j)$ to its \emph{children} $(i\pm1,j+1)$ (whenever these lie in $\mathbb L_{r-1}$). Declare the site $(i,j)$ \emph{open} if the block $B(i,j)$ is nice, and \emph{closed} otherwise. We use the blocks with $L=L_p$ and write $\theta:=1-\eps_d(p)$. Then, by Lemmas~\ref{l_blocks_indep} and~\ref{l_block_nice_prob}, every site is open with probability at least~$\theta$, the states of the sites at different levels are independent, and for every finite set~$F$ of sites
\begin{equation}
\label{eq_closed_set}
\IP[\text{all sites of $F$ are closed}]\le(1-\theta)^{\lceil|F|/2\rceil}.
\end{equation}
For a set $B$ of sites at level~$0$, let $\cR^B$ be the set of sites $v$ for which there is an oriented path $v_0,v_1,\dots,v_n=v$ with $v_0\in B$ such that $v_0,\dots,v_{n-1}$ are open (the last site~$v$ need not be open); in particular $B\subset\cR^B$. Note that if $v\in\cR^B$ is open, then all its children belong to~$\cR^B$.

\begin{lem}
\label{l_comparison}
Let $B$ be a set of sites at level $0$, and assume that at time~$s$, for every $(i,0)\in B$, at least one vertex of~$S_i$ has fitness~$0$. Then for every $(i,j)\in\cR^B$, at least one vertex of $S_i$ has fitness~$0$ at time $s+jL$; if moreover $(i,j)$ is open, then both vertices of $S_i$ have fitness~$0$ at time $s+(j+1)L$.
\end{lem}
\begin{proof}
By induction on~$j$. For $j=0$ the first claim is the assumption. If $(i,j)\in\cR^B$ with $j\ge1$, then it has an open parent $(i',j-1)\in\cR^B$, $i'=i\pm1$, and by the induction hypothesis and Lemma~\ref{l_prop_block}, both vertices of~$S_{i'}$ have fitness~$0$ at time $s+jL$; one of them belongs to~$S_i$. The second claim follows from Lemma~\ref{l_prop_block}.
\end{proof}

\section{Oriented site percolation on a strip}
\label{s_OP_strip}
In this section we consider a random configuration of open and closed sites on the strip $\mathbb L_{r-1}$ of Section~\ref{s_4blocks} (see~\eqref{Ldefined}); the only assumption on it is that, for some $\theta\in(0,1)$, the bound~\eqref{eq_closed_set} holds for every finite set~$F$ of sites. Let $H_j=\{(i,j)\in\mathbb L_{r-1}\}$ be the $j$-th level; note that $\frac{r-1}{2}\le|H_j|\le\frac {r-1}2+1$. For a set~$B\subset H_0$ we use the set $\cR^B$ of reached sites defined before Lemma~\ref{l_comparison}; we write $\cR^a$ for $\cR^{\{a\}}$. Our main interest is the high-density regime where~$\theta$ is close to~$1$. We fix \[\theta_0:=1-3^{-17}.\]

\begin{lem}
\label{l_OP_p2p}
Let $r\ge4$, $K_0\ge2$ and let~$m$ be an integer with $2(r-1)\le m\le K_0 r$. There is a constant $C_1=C_1(K_0)$ such that if $\theta\ge\theta_0$, then for all $a\in H_0$ and $b\in H_m$,
\[\IP[b\notin\cR^a]\le C_1(1-\theta).\]
\end{lem}

\begin{lem}
\label{l_OP_density}
There is an absolute constant $C_3$ such that the following holds. Let $r\ge 5$, $K_0\ge2$ and let~$m$ be an integer with $2(r-1)\le m\le K_0 r$. If $\theta\ge\theta_0$ and $B\subset H_0$ satisfies $|B|\ge\frac12|H_0|$, then for every $s>0$,
\[
\IP\big[|H_m\setminus\cR^B|\ge s\big]\le e^{C_3K_0 r}\big(6^8(1-\theta)\big)^{s/16}.
\]
Consequently, if $h\in(0,1]$ satisfies $h\ln\frac{1}{6^8(1-\theta)}\ge 32(C_3K_0+1)$, then
\begin{equation}
\label{eq_density_cor}
\IP\big[|H_m\setminus\cR^B|\ge h|H_m|\big]\le e^{C_3K_0-r+1}.
\end{equation}
\end{lem}
The second statement follows from the first one since $|H_m|\ge (r-1)/2$. The rest of this section is devoted to the proofs.

\paragraph{Tiles and contours.} We use the standard contour (Peierls) method for two-dimensional oriented percolation, see e.g.\ Section~10 of~\cite{D84}; since we deal with \emph{site} percolation and need some non-standard estimates, we give the details. To each site $v=(i,j)$ we associate its \emph{tile} $T(v)=\{(x,y)\in\R^2:|x-i|+|y-j|\le1\}$. The tiles of all sites of $\{(i,j)\in\Z^2: i+j\text{ even}\}$ tile the plane; the corners of the tiles are the points of $\Z^2$ with $i+j$ odd, and the sides of the tiles, which we call \emph{edges}, are diagonal segments. Each tile has an upper-left, an upper-right, a lower-left and a lower-right edge, and two tiles share an edge if and only if one of the sites is a child of the other; the upper edges of $T(v)$ are shared with the tiles of the children of~$v$. For $m\ge1$, let
\[
Q=Q_m:=\bigcup\big\{T(v): v\in\mathbb L_{r-1},\ v\text{ at level }0,\dots,m\big\}.
\]
The edges on the boundary~$\partial Q$ are the lower edges of the tiles at level~$0$ (the \emph{bottom}), the upper edges of the tiles at level~$m$ (the \emph{top}), and the outer edges of the tiles with $i=0$ or $i=r-1$ (the \emph{walls}); the number of boundary edges is at most $4K_0 r+10$ when $m\le K_0 r$.

Fix $B\subset H_0$ and let $U$ be the union of the tiles of the sites of $\cR^B\cap Q$. The key observation is the following.
\begin{itemize}
\item[($\star$)] If an edge is shared by the tile of $v\in\cR^B$ (below) and the tile of a site $v'\notin\cR^B$ (above; so $v'$ is a child of~$v$), then $v$ is closed.
\end{itemize}
Indeed, if $v\in\cR^B$ were open, its child~$v'$ would belong to $\cR^B$. Let~$X$ be the union of the tiles of $Q$ not belonging to~$U$. An edge shared by a tile of~$U$ and a tile of~$X$ is called \emph{paid} if the $U$-tile lies below it, and \emph{free} otherwise (see Figure~\ref{f_contours}); by~($\star$), the lower site of a paid edge is closed. A closed site has only two upper edges, so a set of~$k$ paid edges involves at least $k/2$ distinct closed sites. The number of possible continuations of a lattice path along edges at each corner, excluding going back, is at most~$3$.

We shall traverse curves made of edges; each edge changes the first coordinate by~$\pm1$. If a curve formed by edges of $\partial X$ is traversed with~$X$ on its left, then an edge separating~$X$ from a region below it is traversed in the direction of increasing first coordinate, and an edge separating $X$ from a region above it in the direction of decreasing first coordinate. In particular, along such a curve, paid edges increase the first coordinate and free edges decrease it. Finally, we shall use the following elementary fact.
\begin{itemize}
\item[($\star\star$)] Let~$V$ be an edge-connected union of tiles of~$Q$, and let~$Y$ be the union of the tiles of an edge-connected component of the set of the remaining tiles of~$Q$. If both~$V$ and~$Y$ contain edges of~$\partial Q$, then the edges shared by a tile of~$V$ and a tile of~$Y$ form a lattice curve (at corners where it touches itself we continue by a fixed rule), the \emph{common boundary} of~$V$ and~$Y$, which has both endpoints on~$\partial Q$. Moreover, $Y$ is the union of the tiles of~$Q$ lying on one side of this curve, so that the edges of~$\partial Q$ belonging to~$Y$ form one of the two arcs into which the endpoints of the curve divide~$\partial Q$.
\end{itemize}
Indeed, a tile not in~$Y$ sharing an edge with a tile of~$Y$ belongs to~$V$, and every other edge-connected component of the tiles of~$Q$ not in~$V$ shares an edge with~$V$; hence the set of the tiles of~$Q$ not in~$Y$ is edge-connected and meets~$\partial Q$. Therefore~$Y$ has no holes, and, by the Jordan curve theorem for polygons, its boundary (with the corners resolved by the fixed rule) is a simple closed lattice curve consisting alternately of arcs of~$\partial Q$ and of pieces made of edges shared with~$V$. Each such piece cuts off from $Q\setminus Y$ the part lying beyond it, so that the edge-connectedness of $Q\setminus Y$ leaves only one piece.

\begin{figure}
\begin{center}
\includegraphics[width=\textwidth]{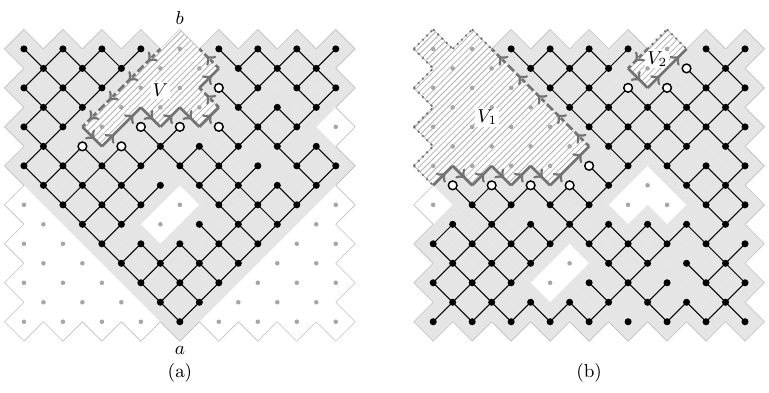}
\caption{Oriented site percolation on a strip (levels $0,\dots,m$, upwards). Filled dots are reached sites, open circles are reached sites which the contour forces to be closed, small grey dots are unreached sites; the shaded diamonds are the tiles of~$U$. Thick grey lines are paid (solid) and free (dashed) edges, oriented so that the unreached region lies on their left; dotted lines are boundary edges of~$Q$. (a)~Proof of Lemma~\ref{l_OP_p2p}: the curve~$\gamma$ separating the component~$V$ (hatched) of~$X$ containing~$T(b)$ from the component~$Y$ of $Q\setminus V$ containing~$T(a)$. (b)~Proof of Lemma~\ref{l_OP_density}: the boundaries~$\Gamma_1,\Gamma_2$ of the components $V_1,V_2$ (hatched) containing unreached top sites.}
\label{f_contours}
\end{center}
\end{figure}

\begin{proof}[Proof of Lemma~\ref{l_OP_p2p}]
Here $B=\{a\}$. Suppose that $b\notin\cR^a$. Let $V$ be the union of the tiles of the edge-connected component of~$X$ containing $T(b)$, and let~$Y$ be the union of the tiles of the edge-connected component of $Q\setminus V$ (as a set of tiles) containing $T(a)$. Every tile of~$Y$ sharing an edge with a tile of~$V$ belongs to~$U$ (a tile of~$X$ sharing an edge with~$V$ would belong to~$V$). Since $T(b)$ has edges on the top and $T(a)$ on the bottom, by~($\star\star$) the common boundary of~$V$ and~$Y$ is a lattice curve~$\gamma$ separating~$V$ from~$Y$ in~$Q$, with both endpoints on~$\partial Q$, and all its edges are paid or free. Let~$k$ be the number of edges of~$\gamma$, orient~$\gamma$ so that~$V$ lies on its left, and let~$D$ be the increment of the first coordinate along~$\gamma$. Then the number of paid edges on~$\gamma$ equals $(k+D)/2$.

We claim that $D\ge0$ unless~$\gamma$ joins the bottom to the top, in which case $k\ge m\ge 2r-2$ and $D\ge-(r+1)$. Recall that $V\ni T(b)$ touches the top and $Y\ni T(a)$ touches the bottom. If both endpoints of~$\gamma$ lie on the top, then~$V$ is the region between~$\gamma$ and the top (the other region contains $T(a)$), and traversing~$\gamma$ with~$V$ on its left means going from the left endpoint to the right one, so $D>0$. If both endpoints lie on the bottom, then~$Y$ is the region between~$\gamma$ and the bottom and, again, $\gamma$ is traversed from left to right. The same argument applies when the endpoints lie on the two different walls, or on a wall and on the top or the bottom: in each case one checks that~$\gamma$ is traversed away from the left wall or towards the right wall, so that $D\ge0$. Both endpoints cannot lie on the same wall, since then one of the two regions would touch neither the top nor the bottom. Finally, if~$\gamma$ joins the bottom to the top, then $k\ge m$, and $D\ge-(r+1)$ since all first coordinates lie in $[-1,r]$.

Consequently, the number of paid edges on~$\gamma$ is at least $k/2$ in the first case, and at least $(k-r-1)/2\ge k/4-1$ in the second case. It is also at least~$1$: otherwise $D=-k$, which is impossible in the first case, and in the second case it would give $m\le k\le r+1<2r-2\le m$ (recall $r\ge4$), a contradiction. Hence~$\gamma$ forces at least $\max(1,\frac k8-\frac12)$ distinct closed sites, all determined by~$\gamma$.

It remains to count the possible curves. In each case above, one endpoint of~$\gamma$ lies on the boundary at distance at most~$k$ from~$a$, from~$b$, or from a corner of~$Q$, except for the curves joining the two walls or the bottom to the top, which have $k\ge r-1$ and whose starting point can be any of the at most $4K_0 r+10\le 20K_0k$ boundary edges. Hence the number of possible (oriented) curves with~$k$ edges is at most $C K_0k3^k$. By~\eqref{eq_closed_set}, a given curve with~$k$ edges occurs with probability at most $(1-\theta)^{\max(1,\frac k{16}-\frac14)}$, so that
\begin{align*}
\IP[b\notin\cR^a]&\le\sum_{k\ge1}CK_0k3^k(1-\theta)^{\max(1,\frac k{16}-\frac14)}\\
& \le CK_0(1-\theta)\Big(\sum_{k\le20}k3^k+3^{20}\sum_{k>20}k\big(3^{16}(1-\theta)\big)^{(k-20)/16}\Big),
\end{align*}
which is at most $C_1(K_0)(1-\theta)$ because $3^{16}(1-\theta)\le\frac13$ for $\theta\ge\theta_0$.
\end{proof}

\begin{proof}[Proof of Lemma~\ref{l_OP_density}]
Let $S=H_m\setminus\cR^B$ and assume $S\neq\emptyset$. Let $V_1,\dots,V_\nu$ be the unions of the tiles of the edge-connected components of~$X$ which contain a tile of a site of~$S$, and let $t_i\ge1$ be the number of such tiles in~$V_i$, so that $\sum_it_i=|S|$. Every edge-connected component~$Y$ of $Q\setminus V_i$ (as a set of tiles) contains a tile sharing an edge with a tile of~$V_i$; this tile belongs to~$U$ (a tile of~$X$ adjacent to~$V_i$ would belong to~$V_i$), and following an oriented path of reached sites back to~$B$ we see that~$Y$ contains the tile of a site of~$B$, whose lower edges belong to the bottom. Consequently, $V_i$ has no holes (i.e., $\R^2\setminus V_i$ is connected), and its boundary is a closed lattice curve~$\Gamma_i$ (it may touch itself at corners, but does not use any edge twice), which we traverse counterclockwise, i.e.\ with~$V_i$ on its left. Its edges are either boundary edges of~$Q$, or edges shared with tiles of~$U$, which are paid or free. Let~$\mu_i$ be the number of paid edges on~$\Gamma_i$. We claim that
\begin{equation}
\label{eq_paid_i}
\mu_i\ge t_i/4\qquad\text{for all } i.
\end{equation}
We distinguish three cases.

\emph{Case 1: $V_i$ contains no tile at level~$0$.} Since~$\Gamma_i$ is closed, the increments of the first coordinate along it sum to zero. The paid edges contribute $+\mu_i$, the free ones $-(\text{their number})$. As for the boundary edges of~$Q$ on~$\Gamma_i$: the two upper edges of each of the~$t_i$ tiles at level~$m$ contribute~$-1$ each, the two outer edges of a tile at a wall (not at level $0$ or~$m$) contribute $+1$ and $-1$, and the corrections at the (at most two) top corners of~$Q$ are at most~$2$ in total. Hence $\mu_i\ge 2t_i-2$. Moreover $\mu_i\ge1$: the lowest tiles of $V_i$ are at a level $\ge1$, and their parents are tiles of~$U$ (they are not in~$V_i$, and not in~$X$), so the corresponding edges are paid. Thus $\mu_i\ge\max(1,2t_i-2)\ge t_i/2$.

\emph{Case 2: $V_i$ contains a tile at level~$0$, and it does not touch both walls.} Say that~$V_i$ has no tile with first coordinate~$r-1$. Add to~$Q$ the (virtual) tiles of the sites with first coordinate $r$ and $r+1$ at levels $0,\dots,m$; they form an edge-connected set touching the top and the bottom and sharing edges only with tiles with first coordinate~$r-1$, hence not with~$V_i$. By~($\star\star$) (applied with $r+1$ in place of~$r-1$), the common boundary of~$V_i$ and the component of the complement of~$V_i$ containing the virtual tiles is a lattice curve~$\gamma$ consisting of edges between~$V_i$ and tiles of~$U$, and it lies on~$\Gamma_i$. Since both regions touch the top and the bottom, $\gamma$ joins the bottom to the top, so it has $k\ge m\ge2r-2$ edges, and traversed with~$V_i$ on its left it has increment of the first coordinate $D\ge-(r+1)$. Hence $\mu_i\ge(k+D)/2\ge (r-3)/2$. Since $t_i\le|H_m|\le\frac {r-1}2+1$ and $r\ge5$, this gives $\mu_i\ge t_i/4$.

\emph{Case 3: $V_i$ contains a tile at level $0$ and touches both walls.} Then no component~$Y$ of the complement of~$V_i$ in~$Q$ touches the top: otherwise~$Y$ would touch the top and the bottom (recall that~$Y$ contains the tile of a site of~$B$) and would separate the two walls, which is impossible since~$V_i$ is connected. By~($\star\star$), the common boundary of~$V_i$ and such a component~$Y$ is a lattice curve~$\gamma_Y$ on~$\Gamma_i$ with endpoints on the bottom or on the walls, but not on both walls (since~$V_i$ contains a tile at level~$0$ and~$Y$ does not touch the top); traversed with~$V_i$ on its left, it goes from left to right. The tiles of the sites of $B\cap Y$ are centred at points of level~$0$ whose first coordinates differ by at least~$2$, and, by~($\star\star$), their lower edges lie on the arc of~$\partial Q$ between the endpoints of~$\gamma_Y$ which belongs to~$Y$ (and hence does not contain the top). If both endpoints of~$\gamma_Y$ are on the bottom, they thus lie to the left and to the right of all these tiles, so that the increment of the first coordinate along~$\gamma_Y$ is $D_Y\ge 2|B\cap Y|$. If one of the endpoints is on a wall, its first coordinate is at most~$0$ (for the left wall) or at least~$r-1$ (for the right wall), and we only obtain $D_Y\ge 2|B\cap Y|-1$. In both cases $D_Y\ge|B\cap Y|$, and $\gamma_Y$ contains at least $D_Y$ paid edges. Summing over~$Y$, we get $\mu_i\ge|B|\ge\frac12|H_0|\ge\frac12(|H_m|-1)\ge\frac12(t_i-1)$, and also $\mu_i\ge1$ since $B\ne\emptyset$. Hence $\mu_i\ge t_i/4$.

This proves~\eqref{eq_paid_i}. The curves~$\Gamma_i$ are edge-disjoint, so the total number of paid edges is $\mu:=\sum_i\mu_i\ge|S|/4$, and the family $\Gamma=(\Gamma_1,\dots,\Gamma_\nu)$ forces at least $\mu/2$ distinct closed sites, all determined by~$\Gamma$; by~\eqref{eq_closed_set}, a given family~$\Gamma$ thus occurs with probability at most $(1-\theta)^{\mu/4}$. Let $\Lambda$ be the total number of edges of~$\Gamma$ and let $\beta$ be the total number of boundary edges of~$Q$ on~$\Gamma$, so that $\beta\le 4K_0 r+10$. Since the increments of the first coordinate along each~$\Gamma_i$ sum to zero, the number of free edges is at most $\mu+\beta$, hence $\Lambda\le 2\mu+2\beta$.

Each $\Gamma_i$ contains an upper edge of a tile at level~$m$; we encode~$\Gamma$ by the set of these starting edges (at most $2^{2|H_m|}\le 2^{r+1}$ choices), by the lengths of the curves (at most $2^\Lambda$ choices given~$\Lambda$) and by their steps (at most $3^\Lambda$ choices). Therefore, writing $\rho=36(1-\theta)^{1/4}\le\frac12$,
\begin{align*}
\IP[|S|\ge s]
&\le 2^{r+1}\sum_{\mu\ge s/4}\ \sum_{\Lambda\le2\mu+2\beta}6^\Lambda(1-\theta)^{\mu/4}\\
&\le 2^{r+1}6^{2\beta}\sum_{\mu\ge s/4}(2\mu+2\beta+1)\rho^{\mu}\\
&\le C\,2^{r-1}6^{2\beta}(s+\beta+1)\rho^{s/4}.
\end{align*}
Since $|S|\le r+1$, we may assume $s\le r+1$, and the prefactor is at most $e^{C_3K_0 r}$ for a suitable absolute constant~$C_3$. Finally, $\rho^{s/4}=(6^8(1-\theta))^{s/16}$.
\end{proof}

\section{Proofs of Theorems~\ref{t_small_p_1_in_x} and~\ref{t_small_p_proportion}}
\label{s_proofs_main}
Throughout this section we work with the continuous-time process. We use the blocks of Section~\ref{s_4blocks} with $L=L_p$, so that~\eqref{eq_closed_set} holds with $\theta=1-\eps_d(p)$, see~\eqref{def_L_eps}. We fix $K_0:=44$ and let $C_1=C_1(K_0)$ and~$C_3$ be the constants of Lemmas~\ref{l_OP_p2p} and~\ref{l_OP_density}. We write $\mathcal F_s$ for the $\sigma$-algebra generated by the graphical construction up to time~$s$; by the Markov property, all the estimates below hold conditionally on~$\mathcal F_s$, for events concerning the rings after time~$s$.

We first record a consequence of Lemma~\ref{l_OP_p2p}: transporting a zero along a chain.
\begin{lem}
\label{l_phase}
Let $(z_0,\dots,z_r)$ be a chain with $r\ge4$, let~$u$ be one of its vertices, and let $n$ be an integer with $2r\le n\le K_0 r$. Let $\tau$ be a stopping time. On the event that at time~$\tau$ a vertex of~$S_i$ has fitness~$0$ for some even~$i$, the conditional probability given $\mathcal F_\tau$ that~$u$ has fitness~$0$ at time $\tau+nL$ is at least $1-(C_1+2)\eps_d(p)$.
\end{lem}
\begin{proof}
Use the blocks with $s=\tau$, and let $a=(i,0)$, where~$i$ is the smallest even index such that~$S_i$ has a vertex with fitness~$0$ at time~$\tau$. Let $u\in S_{i'}$, and let $j\in\{n-2,n-1\}$ be such that $i'+j$ is even, so that $b:=(i',j)$ is a site and $2(r-1)\le j\le K_0r$. By Lemma~\ref{l_OP_p2p} (with $m=j$), $\IP[b\notin\cR^a\mid\mathcal F_\tau]\le C_1(1-\theta)$; the event $\{b\in\cR^a\}$ depends only on the blocks at levels below~$j$, and hence is independent of the event that~$b$ is open, which has probability at least~$\theta$. On $\{b\in\cR^a,\ b\text{ open}\}$, Lemma~\ref{l_comparison} implies that both vertices of~$S_{i'}$, in particular~$u$, have fitness~$0$ at time $\tau+(j+1)L$. If $j=n-2$, then~$u$ keeps fitness~$0$ during the last level except with probability at most $(d+1)pL$. Since $1-\theta=\eps_d(p)$ and $(d+1)pL_p\le\eps_d(p)$, the claim follows.
\end{proof}

\subsection{Proof of Theorem~\ref{t_small_p_1_in_x}}
\label{s_proof_t12}
\paragraph{Graphs with short chains.} Let $\ell_0=\ell_0(d):=2d^3+2$. By~\eqref{diam_ell}, a graph with $\ell(G)<\ell_0$ and maximal degree at most~$d$ has at most $M_d:=1+d+\dots+d^{\ell_0}$ vertices. We apply Lemma~\ref{l_est_pi} with $A=\{\eta:\eta_x=1\}$ and $t:=M_d+1$. If $\eta_x(0)=0$, then $\IP_\eta[\eta_x(t)=1]\le(d+1)pt$ by the observation after~(O). Now let $\eta_x(0)=1$, and let~$E_0$ be the event that no ring in $(0,t]$ carries a~$1$ at any vertex, so that $\IP[E_0^\complement]\le(d+1)M_dt\,p$. On~$E_0$, by~(O), no zero is ever lost, and every active ring gives fitness~$0$ to the whole neighbourhood of the ringing vertex. Let~$y$ be a vertex with fitness~$0$ at time~$0$ if there is one; otherwise, let~$y$ be the vertex whose clock rings first (and a fixed vertex if no clock rings in $(0,1]$), so that on~$E_0$ the vertex~$y$ has fitness~$0$ at time~$1$ if some clock rings in $(0,1]$, which has probability at least $1-e^{-1}$. Let $y=y_0\sim y_1\sim\dots\sim y_k=x$ be a shortest path, so that $k<M_d$. If, in addition, the clock of~$y_l$ rings in $(l+1,l+2]$ for $l=0,\dots,k-1$, then on~$E_0$ the vertices $y_1,\dots,y_k$ successively get fitness~$0$, and $\eta_x(t)=0$. Hence $\IP_\eta[\eta_x(t)=0]\ge(1-e^{-1})^{M_d}-(d+1)M_dt\,p\ge\frac12(1-e^{-1})^{M_d}$ for $p$ small, and Lemma~\ref{l_est_pi} gives $\pi^c(\eta:\eta_x=1)\le C_dp\le C_dp\ln p^{-1}$ for $p\le e^{-1}$, i.e., Theorem~\ref{t_small_p_1_in_x} for graphs with $\ell(G)<\ell_0$.

\paragraph{Graphs with long chains.} Assume now that $\ell:=\ell(G)\ge\ell_0$, and let $\Gamma=(\gamma_1,\dots,\gamma_\ell)$ be a chain of length~$\ell$. The ball of radius~$3$ around any vertex~$z$ has at most $1+d+d^2+d^3<2d^3<\ell$ vertices, so we can fix a vertex $z'\in\Gamma$ at graph distance at least~$4$ from~$z$, and a shortest path $\Pi(z)$ from~$z$ to~$z'$; it is a chain with $r(z)\in[4,\ell-1]$ edges by~\eqref{diam_ell}.

Fix $x\in G$ and a ``source'' vertex $y\in G$, and define $x'=\gamma_\alpha$, $y'=\gamma_\beta$. Among the three vertices $\gamma_1,\gamma_{\lceil\ell/2\rceil},\gamma_\ell$, whose indices differ pairwise by at least $\frac\ell2-1$, at most one has index at distance less than $\frac{\ell-2}4$ from~$\alpha$ (otherwise these two would be at distance less than $\frac\ell2-1$), and similarly for~$\beta$. Hence we can choose one of them, call it~$w$, whose index differs from both $\alpha$ and~$\beta$ by at least $\frac{\ell-2}4\ge4$. Consider the four chains (see Figure~\ref{f_chains})
\begin{align*}
&\Pi_1=\Pi(y)\ (\text{from $y$ to $y'$}), &&\Pi_2\ (\text{the part of $\Gamma$ from $y'$ to $w$}),\\
&\Pi_3\ (\text{the part of $\Gamma$ from $w$ to $x'$}), &&\Pi_4=\Pi(x)\ (\text{reversed, from $x'$ to $x$}).
\end{align*}
Each $\Pi_k$ has $r_k\in[4,\ell-1]$ edges, and $r_2\ge\frac{\ell-2}4$. Let $n^*:=11\ell$, $n_k:=3r_k$ for $k=1,3,4$, and $n_2:=n^*-2-(n_1+n_3+n_4)$. Then $2r_k\le n_k\le K_0r_k$ for all~$k$ (for $k=2$ we use $36\le n_1+n_3+n_4\le9\ell-9$, which gives $2\ell+7\le n_2\le 11\ell-38\le 44\cdot\frac{\ell-2}4$).

\begin{figure}
\begin{center}
\includegraphics{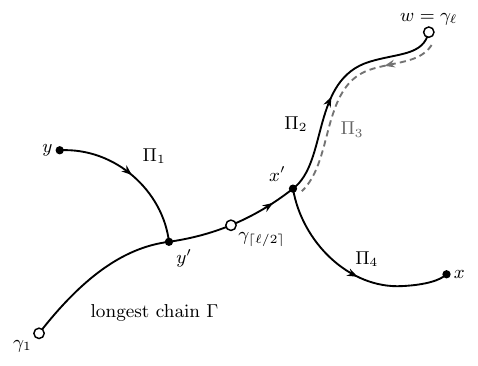}
\caption{Illustration of the proof of Theorem~\ref{t_small_p_1_in_x}: propagating a zero from~$y$ to~$x$ via the longest chain~$\Gamma$, along $\Pi_1$ (from $y$ to $y'$), $\Pi_2$ (along $\Gamma$ from $y'$ to $w$), $\Pi_3$ (back along $\Gamma$ from $w$ to $x'$, dashed) and $\Pi_4$ (from $x'$ to $x$). The open circles are the three candidates for~$w$; here $w=\gamma_\ell$.}
\label{f_chains}
\end{center}
\end{figure}

Let $t:=n^*L$. By Lemma~\ref{l_phase} applied successively to $\Pi_1,\dots,\Pi_4$ (each time from the first vertex of~$\Pi_k$, which belongs to~$S_0$, to the last one) with durations $n_1L,\dots,n_4L$, starting at time~$2L$, we obtain: if $y$ has fitness~$0$ at time~$2L$, then, conditionally on $\mathcal F_{2L}$, the vertex~$x$ has fitness~$0$ at time~$t$ with probability at least $1-4(C_1+2)\eps_d(p)$. We now apply Lemma~\ref{l_est_pi} with $A=\{\eta:\eta_x=1\}$ and this~$t$.

(i) If $\eta_x(0)=0$, take $y=x$; it keeps fitness~$0$ during $[0,2L]$ with probability at least $1-2(d+1)pL$. Hence
\[
\IP_\eta[\eta_x(t)=1]\le\eps':=2(d+1)pL+4(C_1+2)\eps_d(p).
\]
(ii) If $\eta_x(0)=1$, let $y$ be a vertex with fitness~$0$ at time~$0$ if there is one; its zero persists until time $2L$ with probability at least $1-2(d+1)pL\ge\frac12$. If~$\eta$ is the all-ones configuration, the first ring (of any vertex) occurs before time $L/2$ with probability at least $\frac12$; let~$y$ be the ringing vertex (and a fixed vertex if there is no ring before time $L/2$); it receives fitness~$0$ with probability $q\ge\frac12$, which then persists until time~$2L$ with probability at least~$\frac12$. In both cases~$y$ is determined by time $L/2$, and it has fitness~$0$ at time~$2L$ with probability at least $\frac18$. Therefore
\[
\IP_\eta[\eta_x(t)=0]\ge\tfrac18\big(1-4(C_1+2)\eps_d(p)\big)\ge\tfrac1{16}.
\]
Lemma~\ref{l_est_pi} now gives
\[
\pi^c(\eta:\eta_x=1)\le16\eps'=16\big(2(d+1)pL_p+4(C_1+2)\eps_d(p)\big)\le C_dp\ln p^{-1},
\]
for all $p\le p_0(d)$, where~$p_0(d)$ is chosen so that $\theta=1-\eps_d(p)\ge\theta_0$, $L_p\ge2$, $2(d+1)pL_p\le\frac12$ and $4(C_1+2)\eps_d(p)\le\frac12$. Together with the case of short chains, this proves the bound for~$\pi^c$; the bound for~$\pi^d$ follows from~\eqref{eq_pic_pid} in Appendix~\ref{s_app_measures}, decreasing $p_0$ if necessary so that $C_dp\ln p^{-1}\le\frac12$. \qed

\subsection{Proof of Theorem~\ref{t_small_p_proportion}}
\label{s_proof_t13}
Let
\[
h_*=h_*(p,d):=\frac{512(C_3K_0+1)}{\ln\frac{1}{6^8\eps_d(p)}} ,
\]
so that $h_*\le C'/\ln p^{-1}$ for $p$ small. We say that a chain~$\Phi$ is \emph{good} at time~$s$, if at most $h_*|\Phi|$ of its vertices have fitness~$1$ at time~$s$. In what follows, $p_1(d)>0$ is small enough so that $\theta\ge\theta_0$, $h_*\le\frac14$, and several other explicit conditions (stated below) hold for $p\le p_1$; $C_4,C_5,\dots$ denote constants depending only on~$d$.

We refer to the family of chains in the hypothesis of the theorem as \emph{the cover}. Cutting chains into consecutive pieces (sub-paths of chains are chains), which does not change the sum of the lengths, we may and will assume that all chains of the cover have lengths in $[\ell,2\ell)$; their number is then $K\le m|G|/\ell$. If $\ell<\ell_0$ (with~$\ell_0$ as in Section~\ref{s_proof_t12}) the statement is trivial (take $C\ge e^{\ell_0}$), so we assume $\ell\ge\ell_0$. For such a chain~$\Phi$ we have $r=|\Phi|-1\in[\ell-1,2\ell-2]$.

\begin{lem}[Persistence]
\label{l_window}
Let $\Phi$ be a chain of length in $[\ell,2\ell)$, and let $n$ be an integer with $2r\le n\le K_0 r$. Then, for any $s$, on the event that~$\Phi$ is good at time~$s$,
\[
\IP\big[\Phi\text{ is not good at time }s+nL\mid\mathcal F_s\big]\le C_4e^{-\ell}.
\]
\end{lem}
\begin{proof}
Use the blocks of~$\Phi$ starting at time~$s$. Let $B\subset H_0$ be the set of level-$0$ sites~$(i,0)$ such that $S_i$ has a vertex with fitness~$0$ at time~$s$. The sets~$S_i$ of distinct level-$0$ sites are disjoint, so the number of sites of $H_0\setminus B$ is at most half the number of ones on~$\Phi$, i.e., at most $\frac12h_*|\Phi|\le h_* (r-1)\le2h_*|H_0|\le\frac12|H_0|$. Let $k=n-1$, so that $2(r-1)\le k-1<k\le K_0 r$. For $j\in\{k-1,k\}$ let $S^{(j)}=H_j\setminus\cR^B$, and let $Z_j$ be the number of closed sites at level~$j$. Let~$F$ be the number of vertices $v\in \Phi$ such that some ring in $(s+kL,s+(k+1)L]$ carries a~$1$ at~$v$.

Call a site $(i,j)$ with $j\in\{k-1,k\}$ \emph{bad} if it is closed or does not belong to~$\cR^B$. Every vertex~$v$ of~$\Phi$ belongs to~$S_i$ for some site~$(i,k)$ or~$(i,k-1)$. Hence, by Lemma~\ref{l_comparison} and~(O), if~$v$ has fitness~$1$ at time $s+(k+1)L$, then~$v$ belongs to~$S_i$ for some bad site, or~$v$ is one of the~$F$ vertices above: indeed, otherwise either~$v\in S_i$ for a site $(i,k)$ which is not bad, or~$v\in S_i$ for a site $(i,k-1)$ which is not bad, in which case~$v$ has fitness~$0$ at time $s+kL$ and keeps it during the next level. Therefore, the number of ones on~$\Phi$ at time $s+nL$ is at most
\[
2\big(|S^{(k)}|+|S^{(k-1)}|+Z_k+Z_{k-1}\big)+F.
\]
Consequently, the chain~$\Phi$ is good at time $s+nL$ unless one of $|S^{(k)}|,|S^{(k-1)}|,Z_k,Z_{k-1}$ exceeds $h_*(r-1)/16$ or~$F$ exceeds~$h_*(r-1)/2$; we now show that each of these five events has probability at most $C_5e^{-r+1}$.

Since $|H_j|\le r-1$, the event $|S^{(j)}|>h_*(r-1)/16$ implies $|S^{(j)}|\ge\frac{h_*}{16}|H_j|$, and by~\eqref{eq_density_cor} with $h=h_*/16$ (which satisfies its assumption by the choice of~$h_*$), its probability is at most $e^{C_3K_0-r+1}$. Next, split the sites of~$H_j$ into two classes according to the value of~$i$ modulo~$4$; by Lemma~\ref{l_blocks_indep}~(i), the sites of each class are independent, so that the number of closed sites in each class is dominated by a Binomial$(r-1,\eps_d(p))$ variable. Since $\IP[\mathrm{Bin}(a,\rho)\ge b]\le\binom ab\rho^b\le(ea\rho/b)^b$, we obtain
\[
\IP[Z_j> h_*(r-1)/16]\le2\Big(\frac{32e\,\eps_d(p)}{h_*}\Big)^{h_*(r-1)/32}\le 2e^{-r+1},
\]
provided $\frac{h_*}{32}\ln\frac{h_*}{32e\eps_d(p)}\ge1$, which holds for $p$ small since $h_*\ge 512/\ln p^{-1}$ and $\eps_d(p)=O(p\ln p^{-1})$. Finally, for $v\in \Phi$ the event counted in~$F$ has probability at most $(d+1)pL$, and such events are independent for vertices whose indices along~$\Phi$ differ by at least~$3$ (their neighbourhoods are disjoint). Splitting~$\Phi$ into three classes according to the index modulo~$3$ (each of size at most~$r-1$), we get in the same way $\IP[F>h_*(r-1)/2]\le 3\big(6e(d+1)pL/h_*\big)^{h_*(r-1)/6}\le3e^{-r+1}$ for $p$ small. Since $r\ge\ell-1$, this proves the lemma.
\end{proof}

By concatenating windows, Lemma~\ref{l_window} implies that if~$\Phi$ is good at time~$s$ and $n\ge 2r$, then $\Phi$ is good at time $s+nL$ except with conditional probability at most $n\,C_4e^{-\ell}$: indeed, $[s,s+nL]$ can be divided into at most~$n$ windows whose numbers of levels lie in $[2r,4r]$.

\begin{lem}[Spreading from one zero]
\label{l_spread}
Let $\Phi$ be a chain of length in $[\ell,2\ell)$, $v\in \Phi$, and let $n$ be an integer with $2r\le n\le K_0 r$. Let $\tau$ be a stopping time. On the event that~$v$ has fitness~$0$ at time~$\tau$,
\[
\IP\big[\Phi\text{ is good at time }\tau+nL\mid\mathcal F_\tau\big]\ge 1-\delta(p),\qquad
\delta(p):=\frac{(C_1+2)\eps_d(p)}{h_*}\xrightarrow[p\to0]{}0.
\]
\end{lem}
\begin{proof}
Reversing the order of the vertices of~$\Phi$ if necessary, we may assume that $v\in S_i$ for some even~$i$ (the sets~$S_i$ with even~$i$ cover all vertices of~$\Phi$ except possibly the last one). By Lemma~\ref{l_phase}, $\IP[\eta_u(\tau+nL)=1\mid\mathcal F_\tau]\le(C_1+2)\eps_d(p)$ for every $u\in\Phi$. Hence the conditional expectation of the number of ones on~$\Phi$ at time $\tau+nL$ is at most $(C_1+2)\eps_d(p)|\Phi|$, and the claim follows from Markov's inequality. Finally, $\delta(p)\to0$ since $\eps_d(p)=O(p\ln p^{-1})$ and $h_*^{-1}=O(\ln p^{-1})$.
\end{proof}

\begin{lem}[Transfer between chains]
\label{l_attempt}
Let $\Phi\neq \Phi'$ be two chains of the cover. There is an integer $\mathcal D=\mathcal D(\Phi,\Phi')\le \mathcal D_{\max}:=3\ell(G)+10\ell$ with the following property. Let~$\tau$ be a stopping time; on the event that $\Phi$ is good at time~$\tau$,
\[
\IP\big[\Phi'\text{ is good at time }\tau+\mathcal D L\mid\mathcal F_\tau\big]\ge\tfrac12 .
\]
\end{lem}
\begin{proof}
Fix a vertex $v\in \Phi'$. Since $|\Phi|\ge\ell_0>2d^{3}$, there is $u\in \Phi$ at graph distance at least~$4$ from~$v$; let~$P$ be a shortest path from~$u$ to~$v$, a chain with $r_P\in[4,\ell(G)-1]$ edges. Let $r_\Phi=|\Phi|-1$ and $r'=|\Phi'|-1$ be the numbers of edges of~$\Phi$ and~$\Phi'$. We proceed in three phases.

\emph{Phase 1} (on $\Phi$, $n_1=3r_\Phi$ levels). Since $\Phi$ is good at time~$\tau$, the set $B\subset H_0$ defined as in the proof of Lemma~\ref{l_window} satisfies $|H_0\setminus B|\le\frac12|H_0|$, so that some~$S_i$ with even~$i$ has a vertex with fitness~$0$ at time~$\tau$. By Lemma~\ref{l_phase}, $u$ then has fitness~$0$ at time $\tau+n_1L$ except with probability at most $(C_1+2)\eps_d(p)$.

\emph{Phase 2} (on $P$, $n_2=3r_P$ levels). By Lemma~\ref{l_phase}, $v$ then has fitness~$0$ at time $\tau+(n_1+n_2)L$ except with probability at most $(C_1+2)\eps_d(p)$.

\emph{Phase 3} (on $\Phi'$, $n_3=2r'$ levels). By Lemma~\ref{l_spread}, $\Phi'$ is then good at time $\tau+(n_1+n_2+n_3)L$ except with probability at most~$\delta(p)$.

Let $\mathcal D:=n_1+n_2+n_3$. Since $r_\Phi,r'\le2\ell-2$ and $r_P\le\ell(G)-1$, we have $n_1\le6\ell$, $n_2\le3\ell(G)$ and $n_3\le4\ell$, so $\mathcal D\le \mathcal D_{\max}$. The failure probability is at most $2(C_1+2)\eps_d(p)+\delta(p)\le\frac12$ for $p$ small.
\end{proof}

\begin{proof}[Proof of Theorem~\ref{t_small_p_proportion}]
Let $\mathsf{Good}$ be the set of configurations in which all chains of the cover are good, and $A:=\mathsf{Good}^\complement$. If $\eta\in\mathsf{Good}$, then every chain~$\Phi$ of the cover contains at most $h_*|\Phi|$ ones; since every vertex belongs to some chain and the sum of the lengths is at most $m|G|$, the number of ones in~$\eta$ is at most $mh_*|G|$, i.e., $\bar\eta\le mh_*$. So it suffices to bound $\pi^c(A)$. Let $a_*:=\lceil\log_2(64K)\rceil$, $s_0:=10\ell+2$, $s_1:=s_0+5\ell$ and $T:=s_1+a_*\mathcal D_{\max}$ (all in units of levels), and $t:=TL$. We apply Lemma~\ref{l_est_pi} with this~$t$.

\emph{Staying good.} If $\eta\in\mathsf{Good}$, then by the remark after Lemma~\ref{l_window} (applied to each chain, with $n=T\ge 2r$) and the union bound,
\[
\IP_\eta[\eta(t)\in A]\le\eps:=KTC_4e^{-\ell}.
\]

\emph{Becoming good.} Let $\eta$ be arbitrary. As in part~(ii) of the proof of Theorem~\ref{t_small_p_1_in_x}, with probability at least~$\frac18$ there is a vertex~$y$, determined by time $L/2$, having fitness~$0$ at time~$2L$. Let~$\Phi_*$ be a chain of the cover containing~$y$, chosen according to some fixed rule (so that~$\Phi_*$ is also determined by time $L/2$). By Lemma~\ref{l_spread} (with $n=s_0-2=10\ell\in[2r,K_0 r]$), $\Phi_*$ is good at time $s_0L$ with conditional probability at least $1-\delta(p)\ge\frac12$. Denote by $E$ the event that both happen, so that $\IP_\eta[E]\ge\frac1{16}$.

Next, for every chain $\Phi'\neq \Phi_*$ and every $a\in\{1,\dots,a_*\}$, the $a$-th \emph{attempt} for~$\Phi'$ is the procedure of Lemma~\ref{l_attempt} for the pair $(\Phi_*,\Phi')$, started at time $(s_1+a\mathcal D_{\max}-\mathcal D(\Phi_*,\Phi'))L$ and hence ending at time $(s_1+a\mathcal D_{\max})L$. The attempts for a fixed~$\Phi'$ occupy disjoint time intervals. Let~$E_2$ be the event that~$\Phi_*$ is good at all starting times of attempts and at time~$t$. There are at most $Ka_*+1$ such times, all at least $5\ell\ge2r$ levels after~$s_0$, so $\IP_\eta[E\cap E_2^\complement]\le(Ka_*+1)TC_4e^{-\ell}$. By Lemma~\ref{l_attempt} and the Markov property applied successively at the starting times (each attempt ends before the next one starts, and at the start of each attempt, on the event that~$\Phi_*$ is good, its conditional probability of failure is at most~$\frac12$), for each~$\Phi'$ the probability that~$E_2$ occurs and all $a_*$ attempts for~$\Phi'$ fail is at most $2^{-a_*}\le\frac1{64K}$. Finally, if the $a$-th attempt for~$\Phi'$ succeeds, then $\Phi'$ is good at time $(s_1+a\mathcal D_{\max})L$, and it is still good at time~$t$ except with probability at most $TC_4e^{-\ell}$ (if $a<a_*$, the remaining time is at least $\mathcal D_{\max}\ge2(|\Phi'|-1)$ levels). Summing over $a$ and~$\Phi'$, we obtain
\[
\IP_\eta[\eta(t)\in\mathsf{Good}]\ge\frac1{16}-\frac1{64}-3Ka_*TC_4e^{-\ell}\ge\frac1{32},
\]
provided that $3Ka_*TC_4e^{-\ell}\le\frac1{64}$.

\emph{Conclusion.} Since each of $\ell$ and $\ell(G)$ is at most $|G|$ and $K\le m|G|$, we have $T\le30a_*|G|$ and $a_*^2\le\big(7+\log_2(m|G|)\big)^2\le 64m|G|$, so that
\[
Ka_*T\le C_6m^2|G|^3.
\]
If $192C_4C_6m^2|G|^3e^{-\ell}\le1$, then the above condition holds, and Lemma~\ref{l_est_pi} gives
\[
\pi^c(A)\le32\eps\le32C_4C_6m^2|G|^3e^{-\ell};
\]
otherwise, the bound $\pi^c(A)\le 192C_4C_6m^2|G|^3e^{-\ell}$ is trivial, since its right-hand side exceeds~$1$. This proves the statement for~$\pi^c$, with $C'$ from the bound $h_*\le C'/\ln p^{-1}$; the bound for~$\pi^d$ follows from Appendix~\ref{s_app_measures} and Theorem~\ref{t_small_p_1_in_x}.
\end{proof}

\section{Proof of Theorem~\ref{t_extinct_zeros}}
\label{s_proof_t14}
For the proof of Theorem~\ref{t_extinct_zeros} we need the following technical result. While it is probably known, for the sake of completeness, we present its proof.
\begin{lem}
\label{elemlemma}
Let $X_n$ be an ergodic Markov chain on a finite state space $\mathcal{S}$ with transition probabilities $p_{xy}$ and stationary distribution $\pi$; assume that $0$ is one of the states. Suppose there is a non-negative function $f:\mathcal{S}\to\R$ such that
\begin{align}
\label{eq_lyap_cond}
\begin{split}
f(0)&=0;\\
\max_{x\in\mathcal{S}: \ p_{0x}>0} f(x)&\le c_0;\\
|f(X_{n+1})-f(X_n)|&\le C;\\
\IE[f(X_{n+1})-f(X_n)\mid X_n=x]&\le-\eps,\quad x\ne 0,
\end{split}
\end{align}
for some $\eps>0$, $c_0>0$, and $C\ge\eps$.
Then there are $c_1,c_2>0$ depending on $c_0,C,\eps$ only, such that
\[
\pi(A_k)\le c_1 e^{-c_2 k},
\quad\text{where }A_k=\big\{x\in\mathcal{S}:\ f(x)\ge k\big\},\ k>0.
\]
\end{lem}
\begin{proof}
Let $\gamma:=\frac{\eps}{2C^2}$ and $V:=e^{\gamma f}$, and write $PV(x):=\sum_{y}p_{xy}V(y)$. For $x\neq0$, with $Z=f(X_{n+1})-f(X_n)$ and $e^z\le1+z+\frac{z^2}2e^{|z|}$, we have (cf.\ Theorem~2.1.7 in~\cite{FMM} or Lemma~4 in~\cite{VOLK})
\[
\frac{PV(x)}{V(x)}=\IE\big[e^{\gamma Z}\mid X_n=x\big]\le1-\gamma\eps+\frac{\gamma^2C^2}{2}e^{\gamma C}\le 1-a,
\qquad a:=\frac{\eps^2}{4C^2},
\]
since $\gamma C\le\frac12$ and $e^{1/2}\le2$. Also, $V(0)=1$ and $PV(0)\le e^{\gamma c_0}$. Therefore, by stationarity,
\[
\pi(0)+\sum_{x\neq0}\pi(x)V(x)=\sum_{x}\pi(x)PV(x)\le \pi(0)e^{\gamma c_0}+(1-a)\sum_{x\neq0}\pi(x)V(x),
\]
so that $a\sum_{x\neq0}\pi(x)V(x)\le e^{\gamma c_0}-1$. Since $0\notin A_k$ and $V\ge e^{\gamma k}$ on~$A_k$, we obtain $\pi(A_k)\le a^{-1}(e^{\gamma c_0}-1)e^{-\gamma k}$.
\end{proof}

\begin{proof}[Proof of Theorem~\ref{t_extinct_zeros}]
We apply Lemma~\ref{elemlemma} to the discrete-time chain $(\eta(t))_{t\ge0}$, with the all-ones configuration playing the role of the state~$0$. Let us think of zeros in the configuration $\eta(t)$ as ``particles''. A {\it type~$1$} particle is a particle which is ``alone'' (has no neighbouring particles), and a {\it type~$2$} particle is a particle which has at least one neighbouring particle. Let $N^{(1)}_t$ and $N^{(2)}_t$ be the numbers of particles of type~$1$ and type~$2$ at time~$t$, and let $N_t=N^{(1)}_t+N^{(2)}_t$ be the total number of zeros. We consider the function
\[
f(\eta(t)) = N^{(1)}_t + (1-h) N^{(2)}_t=N_t-h N^{(2)}_t
\]
with some $h=h(q)\in(0,1)$ to be chosen later; intuitively, type~$2$ particles have a smaller ``weight'' than those of type~$1$. We will prove that for every configuration~$\eta$ other than the all-ones configuration
\begin{equation}
\label{eq_Lyapunov}
\IE\big[f(\eta(t+1))-f(\eta(t))\mid \eta(t)=\eta\big]\le-\eps
\end{equation}
for some $\eps=\eps(q,d)>0$. The other conditions in~\eqref{eq_lyap_cond} hold with $c_0=d+1$ (from the all-ones configuration at most $d+1$ zeros are created) and $C=1+d^2+d+1$ (an update changes the fitness of at most $d+1$ vertices and the type of at most $1+d^2$ particles). Since $f\ge(1-h)N_t$, Lemma~\ref{elemlemma} then gives~\eqref{eq_extinct}.

Denote by~$v$ the site where the update at time~$t$ occurs, and let $j=\deg(v)\in\{1,\dots,d\}$. We introduce the following notation for this update.
\begin{itemize}
\item Let $M$ be the number of particles at the neighbours of~$v$.
\item Let $w=1$ ($w=1-h$ respectively) if $v$ was of type~$1$ (type~$2$ respectively), that is, $w$ is the weight of the updated particle in our Lyapunov function. Note that $w=1$ if and only if $M=0$.
\item Let~$Y_1$ and~$Y_2$ be the numbers of type~$1$ and type~$2$ particles in $N[v]$ after the update.
\item We denote by~$Z$ the number of particles outside of~$N[v]$ which change their type from $2$ to $1$ because of the update. Such transitions require the removal of a particle at a neighbour of~$v$, so they may only occur when $v$ is of type~$2$. The reverse transitions (from type $1$ to type $2$) may occur as well; however, such a transition decreases the weight of the particle, and we only need an {\em upper bound} in~\eqref{eq_Lyapunov}.
\end{itemize}
Then
\begin{align}
\begin{split}
 f(\eta(t+1))-f(\eta(t)) &\leq Y_1 + (1-h)Y_2  + hZ - (1-h)M - w\\
 &=  Y_1 + Y_2 -hY_2 + hZ  - (1-h)M - w.
\end{split}\label{incr_f}
\end{align}
We denote by $\IE_{v,\eta}$ the expectation given that the current configuration is~$\eta$ and the update occurs at~$v$. Whether $v$ is of type $1$ or $2$, the expected number of particles in $N[v]$ after the update is
\begin{equation}\label{upd1_total}
 \IE_{v,\eta} \big[Y_1 + Y_2\big] = q(j+1).
\end{equation}
Next, a neighbour of~$v$ is of type~$2$ after the update with probability at least~$q^2$ (it suffices that both it and~$v$ are resampled to zero), and $v$ itself is of type~$2$ after the update with probability exactly~$q(1-(1-q)^{j})$. Hence, whatever the type of~$v$,
\begin{equation}\label{upd1_type2}
 \IE_{v,\eta} \big[Y_2\big] \geq jq^2+q(1-(1-q)^{j}).
\end{equation}
When $v$ is of type $1$, $M=0$, $w=1$ and $Z=0$. Hence, from~\eqref{incr_f}--\eqref{upd1_type2},
\begin{equation}\label{eq_Lyapunov_type1}
\IE_{v,\eta}\big[f(\eta(t+1))-f(\eta(t))\big]
 \leq \psi(j) - 1,\qquad\text{where }\psi(j):=q(j+1) - h\big(jq^2+q(1-(1-q)^j)\big).
\end{equation}
Next, assume that $v$ is of type $2$, and exactly~$M=m\ge 1$ neighbours  of~$v$ have particles.
A particle outside of $N[v]$ can change its type from~$2$ to~$1$ only if it is adjacent to one of these~$m$ particles, and each of them has at most $d-1$ neighbours other than~$v$; moreover, the particle at the neighbour of~$v$ must be removed, which has probability $1-q$. Hence
\begin{align}\label{upd2_transitions}
\IE_{v,\eta} \big[Z\big] &\leq m(1-q)(d-1).
\end{align}
Assume also that
\begin{equation}
\label{cond_h_take_m=1}
h<\frac{1}{q+d-dq}.
\end{equation}
Then, \eqref{incr_f}--\eqref{upd1_type2} and~\eqref{upd2_transitions} imply (recall that $m\geq 1$)
\begin{align}\label{eq_Lyapunov_type2}
\begin{split}
\lefteqn{
\IE_{v,\eta}\big[f(\eta(t+1))-f(\eta(t))\big]
}
\\
&\qquad \leq \psi(j)+hm(1-q)(d-1) -(1-h) (m+1)
 = \psi(j) -m\big(1-h(q+d-dq)\big) -1 + h \\
&\qquad\leq \psi(j) -\big(1-h(q+d-dq)\big) -1 + h,
\end{split}
\end{align}
where we used that $1-h(q+d-dq)>0$ by~\eqref{cond_h_take_m=1}.

The right-hand sides of~\eqref{eq_Lyapunov_type1} and~\eqref{eq_Lyapunov_type2} depend on~$j$ only through $\psi(j)$, and for $h<1$ and $q\le\frac12$,
\[
\psi(j+1)-\psi(j)=q-hq^2-hq^2(1-q)^j\ge q(1-2hq)>0 .
\]
Hence both right-hand sides are maximal for $j=d$, where they equal
\[
q(d+1) - h\big(dq^2+q(1-(1-q)^d)\big) - 1
\quad\text{and}\quad
q(d+1) - 2 +h\big(1+d(1-q-q^2)+q(1-q)^d\big),
\]
respectively. Now, the left-hand side of~\eqref{eq_Lyapunov} is bounded above by the corresponding average of the right-hand sides of~\eqref{eq_Lyapunov_type1} and~\eqref{eq_Lyapunov_type2} over the choice of~$v$; therefore, it suffices to choose~$h\in (0,1)$ in such a way that both expressions above are strictly negative, together with~\eqref{cond_h_take_m=1}. That is, for $q\le\frac12$ we need that
\begin{align*}
h >T_1(q,d) &:= \frac{q(d+1)-1}{dq^2+q(1-(1-q)^d)}\quad \text{and}\quad
h < T_2(q,d):=\frac{2-q(d+1)}{1+d(1-q-q^2)+q(1-q)^d}
\end{align*}
(note that the denominator of $T_2$ is positive for $q\le\frac12$), and $h<\frac{1}{q+d-qd}$; moreover, $T_2(q,d)<1$ for $q\le\frac12$, since the difference between the denominator and the numerator of~$T_2$ equals $d-1+q-dq^2+q(1-q)^d>0$. Note that $T_1(q,d)\le0$ for $q\le\frac1{d+1}$, while $T_2(q,d)$ and the right-hand side of~\eqref{cond_h_take_m=1} are positive for such~$q$. Define
\begin{equation}
\label{def_q0}
 q_0=q_0(d)
  = \sup\Big\{q'\in\big(0,\tfrac12\big]: T_1(q,d) <
  \min\big(T_2(q,d),\tfrac{1}{q+d-qd}\big)\ \text{for all } q\in(0,q']\Big\}.
\end{equation}
By continuity, $q_0(d)>\frac1{d+1}$, and for each $q\in (0,q_0)$ there exists~$h\in (0,1)$ such that~\eqref{eq_Lyapunov} holds with some $\eps=\eps(q,d)>0$.
\end{proof}

\appendix
\section{Relation between the stationary measures for the discrete and continuous models}
\label{s_app_measures}
Recall that $\pi^d$ and $\pi^c$ denote the invariant measures for the discrete and continuous models, respectively.
Let $r(\eta)$ be the number of zeros in the configuration $\eta$ if it contains at least one zero, and $r(\eta)=|G|$ if $\eta$ is the all-ones configuration. In the continuous-time process, active rings occur at rate $r(\eta)$, and at each of them the updated vertex is uniform among the active ones; the sequence of configurations after the successive active rings is the discrete-time chain. By a standard argument for the embedded chain of a continuous-time Markov chain,
\[
\pi^d(\eta)=\frac{r(\eta)\pi^c(\eta)}{\sum_\xi r(\xi)\pi^c(\xi)},
\]
where $\xi$ runs over all possible configurations. Let $g:=\max_{x\in G}\pi^c(\eta:\eta_x=1)$. Then $\IE_{\pi^c} [r(\eta)]\ge \IE_{\pi^c}[\text{number of zeros}]\ge |G|(1-g)$, while $r\le|G|$. Hence, for any event~$A$,
\begin{equation}
\label{eq_pic_pid}
\pi^d(A)=\frac{\IE_{\pi^c}[r(\eta)1_A]}{\IE_{\pi^c}[r(\eta)]}
\le \frac{|G|\pi^c(A)}{|G|(1-g)}
= \frac{\pi^c(A)}{1-g}\le 2\pi^c(A)\quad\text{if } g\le\tfrac12.
\end{equation}

\begin{remark}
\label{r_pic_zeros}
In the regime of Theorem~\ref{t_extinct_zeros}, the bound~\eqref{eq_extinct} can be transferred from~$\pi^d$ to~$\pi^c$ as follows. Let $\nu(\eta)$ be the number of zeros of~$\eta$. Since $\pi^c(\eta)$ is proportional to $\pi^d(\eta)/r(\eta)$, and $r(\eta)=\nu(\eta)>k$ on $\{\nu>k\}$, we have
\[
\pi^c(\nu>k)=\frac{\IE_{\pi^d}[r(\eta)^{-1}1_{\{\nu>k\}}]}{\IE_{\pi^d}[r(\eta)^{-1}]}
\le \frac{\pi^d(\nu>k)}{k\,\IE_{\pi^d}[r(\eta)^{-1}]}.
\]
To bound the denominator from below, note that from the all-ones configuration~$\mathbf 1$ the discrete-time chain stays at~$\mathbf 1$ with probability $|G|^{-1}\sum_v p^{\deg(v)+1}\le p^2$ (all degrees are at least~$1$, since~$G$ is connected and has at least two vertices), so that the stationarity equation at~$\mathbf 1$ gives $\pi^d(\mathbf 1)\le p^2\pi^d(\mathbf 1)+1-\pi^d(\mathbf 1)$, i.e., $\pi^d(\nu\ge1)\ge\frac{1-p^2}{2-p^2}=:2\gamma$. Choose $k_1=k_1(d,q)$ such that $c_1e^{-c_2k_1}\le\gamma$; then, by~\eqref{eq_extinct}, $\pi^d(1\le \nu\le k_1)\ge\gamma$, and hence $\IE_{\pi^d}[r(\eta)^{-1}]\ge\gamma/k_1$. Therefore
\[
\pi^c(\nu>k)\le \frac{k_1c_1}{\gamma k}\,e^{-c_2k}\qquad\text{for all }k>0.
\]
Since also $\pi^c(\nu>k)\le1$, this gives $\pi^c(\nu>k)\le c_1'e^{-c_2k}$ for all $k>0$, with $c_1'=\max\big(k_1c_1\gamma^{-1},e^{c_2}\big)$.
\end{remark}

\subsection*{Acknowledgment}
The research of S.V.\ is partially supported by the Swedish Research Council grant VR 2019-04173.
S.P.\ was partially supported by CMUP, member of LASI, which is financed by national funds through FCT --- Funda\c{c}\~ao para a Ci\^encia e a Tecnologia, I.P.,
under the project with reference UID/00144/2025,
\url{https://doi.org/10.54499/UID/00144/2025}.

Large language models were used for language editing, for expanding some technical arguments, and for preliminary checks of internal consistency. All mathematical ideas and proof strategies are due to the authors, who reviewed the resulting text and take full responsibility for the content of the paper.

\end{document}